\documentclass[12pt,a4paper]{amsart}
\usepackage{amssymb}
\usepackage{mathrsfs}
\headsep=1cm \numberwithin{equation}{section}
\newtheorem{thm}{Theorem}[section]
\newtheorem{cor}[thm]{Corollary}
\newtheorem{lem}[thm]{Lemma}
\newtheorem{prop}[thm]{Proposition}
\theoremstyle{definition}
\newtheorem{defn}[thm]{Definition}
\theoremstyle{remark}
\newtheorem{rem}[thm]{Remark}
\numberwithin{equation}{section}

\newcommand\Supp{\operatorname{Supp}}
\newcommand\Ass{\operatorname{Ass}}
\newcommand\mAss{\operatorname{mAss}}

\newcommand\Ann{\operatorname{Ann}}
\newcommand\Spec{\operatorname{Spec}}
\newcommand\Rad{\operatorname{Rad}}

\newcommand\Ht{\operatorname{ht}}

\newcommand\depth{\operatorname{depth}}
\newcommand\D{\operatorname{D}}
\begin{document}\title[Locally unmixed modules and ideal topologies]
{Locally unmixed modules and linearly equivalent ideal topologies}
\author[Mona Bahadorian,  Monireh Sedghi and  Reza Naghipour ]{Mona Bahadorian,  Monireh Sedghi and  Reza Naghipour$^*$ \\\\\\\,}
\vspace*{0.5cm}
\address{Department of Mathematics, Azarbaijan Shahid Madani University, Tabriz, Iran.}
\email{mona.bahadorian@gmail.com}

\address{Department of Mathematics, Azarbaijan Shahid Madani University, Tabriz, Iran.}
\email {m\_sedghi@tabrizu.ac.ir}
\email {sedghi@azaruniv.ac.ir}

\address{Department of Mathematics, University of Tabriz, Tabriz, Iran.}
\email{naghipour@ipm.ir} \email {naghipour@tabrizu.ac.ir}

\thanks{ 
$^*$Corresponding author: e-mail: {\it naghipour@ipm.ir} (Reza Naghipour)}%
\begin{abstract}
Let $R$ be a commutative Noetherian ring, and let $N$ be a non-zero finitely generated $R$-module. The purpose of this paper is to show that $N$ is locally unmixed if and only if, for any $N$-proper ideal $I$ of $R$ generated by $\Ht_N I$ elements, the topology defined by $(IN)^{(n)}$, $n \geq 0$, is linearly equivalent to the $I$-adic topology.
\end{abstract}
\maketitle
\section{Introduction}
Let $R$ denote a commutative Noetherian ring, $I$ an ideal of $R$ and $N$ a non-zero finitely generated $R$-module. We denote by  $R[It]$ (resp. $R[It, u]$) the {\it graded  ordinary} (resp. {\it extended}) {\it Rees ring} $\oplus _{n \in \mathbb{N}_0} I^n t^n$  (resp. $\oplus _{n \in \mathbb{Z}} I^n t^n$) of $R$  with respect to $I$, where $t$ is an indeterminate and $u=t^{-1}$.  Also, the {\it graded ordinary Rees module} $\oplus _{n \in \mathbb{N}_0} I^n N$ over $R[It]$  (resp. {\it graded extended Rees module} $\oplus _{n \in \mathbb{Z}} I^n N$ over $R[It, u]$)  is denoted by $N[It]$ (resp. $N[It, u]$), which is  finitely generated.  For any multiplicatively closed subset $S$ of $R$, the $n$th $(S)$-symbolic power of $I$ with respect to  $N$, denoted by $S(I^n N)$, is defined to be the union of $I^n N:_N s$ where $s$ varies in $S$. The $I$-adic filtration $\{ I^n N \}_{n \geq 0}$ and the $(S)$-symbolic filtration $\{ S(I^n N) \}_{n \geq 0}$ induce topologies on $N$ which are called the $I$-{\it adic topology} and the $(S)$-{\it symbolic topology}, respectively. These two topologies are said to be {\it linearly equivalent} if, there  is an integer ${k \geq 0}$ such that $S(I^{n+k} N) \subseteq I^n N$ for all integers $n$. In particular, if $S=R\setminus \bigcup \{ \frak p \in \mAss_R N/ IN \}$, where $ \mAss_R N/ IN$ denotes the set of minimal prime ideals of $ \Ass_R N/ IN$, the $n$th $(S)$-symbolic power of $I$ with respect to $N$, is denoted by $(IN)^{(n)}$, and the topology defined by the filtration $\{ (IN)^{(n)} \}_{n \geq 0}$ is called the {\it  symbolic topology}. The purpose of this paper is to show that $N$ is locally unmixed if and only if, for each $N$-proper ideal $I$ that is generated by $\Ht _N I$ elements, the $I$-adic and the symbolic topologies are linearly equivalent.

 P. Schenzel has characterized unmixed local rings \cite[Theorem 7]{sc1} in terms of comparison of the topologies defined by certain filtrations. Also, D. Katz \cite[Theorem 3.5]{ka} and J. Verma \cite[Theorem 5.2]{ve1} have proved a characterization of locally unmixed rings in terms of $s$-ideals. Equivalence of $I$-adic topology and $(S)$-symbolic topology has been studied, in the case $N=R$, in \cite{ka, ra1,sc1,sc2,sc3}, and has led to some interesting results.

Let $\frak p \in \Supp(N)$. Then $N$-height of $\frak p$, denoted by $\Ht_N \frak p$, is defined to be the supremum of lengths of chains of prime ideals of $\Supp(N)$ terminating with $\frak p$. We have $\Ht_N \frak p=\dim _{R_{\frak p}} N_{\frak p}$. We shall say an ideal $I$ of $R$ is $N$-proper if $N/IN \neq 0$, and, when this is the case, we define the $N$-height of $I$ (written $\Ht_N I$) to be
\begin{center}
$\inf \{ \Ht_N \frak p:\, \frak p \in \Supp(N) \cap V(I) \}$
\end{center}
\begin{flushright}
$(=\inf \{ \Ht_N \frak p:\, \frak p \in \Ass_R (N/IN) \}).$
\end{flushright}

If $(R, \frak m)$ is local, then $\widehat{R}$ (resp. $\widehat{N}$) denotes the completion of $R$ (resp. $N$) with respect to the $\frak m$-adic topology. In particular, for any $\frak p \in \Spec (R)$, we denote $\widehat{R_{\frak p}}$ and $\widehat{{N_{\frak p}}}$ the $\frak p R_{\frak p}$-adic completion of $R_{\frak p}$ and $N_{\frak p}$, respectively.
Then $N$ is said to be an {\it unmixed module} if for any $\frak p \in \Ass _{\widehat{R}} \widehat{N}$, $\dim \widehat{R}/\frak p = \dim N$. More generally, if $R$ is not necessarily local and $N$ is non-zero finitely generated, $N$ is a {\it locally unmixed module} if for any $\frak p \in \Supp(N)$, $N_{\frak p}$ is an unmixed $R_{\frak p}$-module.

  As the main result of this paper we characterize the locally unmixed property of a non-zero finitely generated  $R$-module $N$ in terms of the linearly equivalence of the topologies defined by $\{ I^n N \}_{n \geq 0}$ and $\{ (IN)^{(n)} \}_{n \geq 0}$, for certain $N$-proper ideals $I$ of $R$. More precisely we shall show that:
\begin{thm}
Let $R$ be a Noetherian  ring and $N$ a non-zero finitely generated  $R$-module.  Then the following conditions are equivalent:

$\rm(i)$ $N$ is locally unmixed.

$\rm(ii)$ For each $N$-proper ideal $I$ of $R$ that is generated by $\Ht_N I$ elements, the topology given by $\{ (IN)^{(n)} \}_{n \geq 0}$ is linearly equivalent to the $I$-adic topology on $N$.
\end{thm}

One of our tools for proving Theorem {\rm1.1} is the following, which plays a key role in this paper.  Recall that a prime ideal $\frak{p}$ of $R$ is called
a {\it quitessential prime ideal of }$I$ with respect to $N$  precisely when there exists
$\frak{q}\in \Ass_{\widehat{R}_\frak{p}}\widehat{N}_\frak{p}$ such that
$\Rad(I\widehat{R}_\frak{p}+ \frak{q})= \frak{p}\widehat{R}_\frak{p}$. The set of
quitessential primes of $I$ is denoted by $Q(I, N)$. Then,  the set of  {\it essential primes  of} $I$ with respect to $N$, denoted by $E(I, N)$, is defined to be the
set $\{\frak{q}\cap R\mid\, \frak{q}\in Q(uR[It, u],  N[It, u])\}$.  \\

\begin{thm}
Let $R$ denote a Noetherian ring, $N$ a non-zero finitely generated $R$-module and $I$ a $N$-proper ideal of $R$ such that $E(I,N)= \mAss _R N/IN$. Then, the $I$-adic topology $\{ I^n N \}_{n \geq 0}$ and the topology defined by $\{ (I^n N)^{(n)} \}_{n \geq 0}$ are linearly equivalent.
\end{thm}

The proof of Theorem {\rm1.2} is given in {\rm1.13}.\\

Throughout this paper, $R$ will always be a commutative Noetherian
ring with non-zero identity, $N$ will be a non-zero finitely
generated $R$-module, and $I$ will be an $N$-proper ideal of $R$,
i.e., $N/IN \neq0$. For each $R$-module $L$, we denote by $\mAss_RL$
the set of minimal primes of $\Ass_RL$.    For
any ideal $J$ of $R$, {\it the radical of} $J$, denoted by
$\Rad(J)$, is defined to be the set $\{x\in R \,: \, x^n \in J$ for
some $n \in \mathbb{N}\}$. For any unexplained notation and
terminology we refer the reader to \cite{bh} or \cite{n}.

\section{The Results}
 The main result of this section is to show that  a non-zero finitely generated module $N$ over a Noetherian ring $R$  is locally unmixed if and only if,  for any $N$-proper ideal $I$ of $R$ that can be generated by $\Ht_N I$ elements,  the topologies defined by $\{ I^n N \}_{n \geq 0}$ and $\{ (IN)^{(n)} \}_{n \geq 0}$, on $N$, are linearly equivalent.  We begin with the following remark.

\begin{rem}
Let $R$ be a Noetherian ring and $N$ a finitely generated  $R$-module.  For a submodule $M$ of $N$ and an ideal $I$ of $R$, the increasing sequence of submodules
\begin{center}
$M \subseteq M:_N I \subseteq M:_N I^2 \subseteq \cdots \subseteq M:_N I^n \subseteq \cdots$
\end{center}
becomes stationary. Denote its ultimate constant value by $M:_N \langle I \rangle$. Note that $M:_N \langle I \rangle= M:_N I^n$ for all large $n$. Let
$$M=Q_1 \cap \cdots \cap Q_r \cap Q_{r+1} \cap \cdots \cap Q_s$$ be an irredundant primary decomposition of $M$, with $I \subseteq \Rad (Q_i :_R N)$, exclusively for $r+1 \leq i \leq s$. Then, from the definition, it easily follows that $M:_N \langle I \rangle = Q_1 \cap\dots \cap Q_r$. Therefore
\begin{center}
$\Ass_R {N}/{(M:_N \langle I \rangle)} = \{ \frak p \in \Ass_R {N}/{M•} : I \nsubseteq \frak p \} = \Ass_R ({N}/{M•}) \setminus V(I)$.
\end{center}
\end{rem}

Now we can state and prove the following lemma. Here $D_I(L)$ denotes the ideal transform of the $R$-module $L$ with respect to an ideal $I$
of $R$ (see \cite[2.2.1]{bs}).
\begin{lem}
Let $(R, \frak m)$ be local (Noetherian) ring, $I$ an ideal of $R$ and $N$ a non-zero finitely generated $R$-module  such that $\depth N > 0$. Then, for all integers $n\geq 0$, we have $$I^n N:_N \langle{\frak m}\rangle \subseteq \D_{\frak m} (I^n N).$$
\end{lem}

\begin{proof}
The assertion follows  from \cite[Corollary 2.2.18]{bs} and the fact that $\depth I^n N > 0$ for all integers $n\geq 0$.
\end{proof}

The next result concerns the associated prime ideals of the Rees module $ N[It]$ for a non-zero finitely generated module $N$ over a Noetherian ring
$R$ and an ideal $I$ in $R$.
\begin{prop}
Let $R$ be a Noetherian ring, $I$ an ideal of $R$ and $N$ a non-zero finitely generated $R$-module. Then
\begin{center}
$\Ass_{R[It]} N[It] = \{ \oplus_{n\geq 0} (I^n \cap \frak p):\, \frak p \in \Ass_R N\}$.
\end{center}
\end{prop}

\begin{proof}
Let $ \frak q \in \Ass_{R[It]} N[It]$. Then in view of \cite[Lemma 1.5.6]{bh} there exists a homogenous element $x$ of $N[It]$ such that
$\frak q= \Ann_{R[It]} x$. Suppose that $x \in I^vN$ for some integer $v \geq 0$. Then we have
\begin{center}
$\frak q = (0:_{R[It]} x)= \oplus _{n \geq 0} (0:_R x) \cap I^n$.
\end{center}

Now, it is easy to see that $\frak p := (0:_R x)$ is a prime ideal of $R$ and so $\frak p \in \Ass_R N$. Hence $\frak q = \oplus_{n\geq 0} (I^n \cap \frak p)$ for some $\frak p \in \Ass_R N$. Conversely, let $\frak p \in \Ass_R N$ and $\frak p = (0:_R x)$ for an element $x \in N$. Then
\begin{center}
$\frak q:=(0:_{R[It]} x)=\oplus_{n\geq 0} (I^n \cap \frak p)$
\end{center}
is a prime ideal of $\Ass_{R[It]} N[It]$,  because $R[It] / \frak q \cong R/{\frak p}[(I+\frak p /\frak p)t]$ is a domain.
\end{proof}

\begin{defn}
Let $R$ be a Noetherian ring and $N$ an $R$-module. A decreasing sequence $\{N_n\}_{n \geq 0}$ of submodules of $N$ is called a {\it filtration} of $N$. If $I$ is an ideal of $R$, then the filtration $\{N_n\}_{n \geq 0}$ is called $I$-{\it filtration} whenever $IN_n \subseteq N_{n+1}$ for all integers $n \geq 0$.
\end{defn}
\begin{lem}
Let $R$ be a Noetherian ring, $I$ an ideal of $R$ and $N$ an $R$-module. Let $\{N_n\}_{n \geq 0}$ be an $I$-filtration of submodules of $N$ such that the ordinary  Rees module $N[It]$ is finitely generated over $R[It]$. Then there exists an integer $k$ such that $N_{n+k}=I^n N_k$, for all integers $n\geq 0$.
\end{lem}

\begin{proof}
The result follows easily  from \cite[Lemma 2.5.4]{et}.
\end{proof}

\begin{cor}
Let $(R, \frak m)$ be a local (Noetherian) ring and $I$ an ideal of $R$. Let $N$ be an $R$-module and set $N_n = I^n N:_N \langle\frak m\rangle$  for each integer $n \geq 0$. Suppose that the module $\oplus _{n\geq 0} N_n$ is finitely generated over the ordinary Rees ring $R[It]$. Then there is an integer $k$ such that $I^{n+k}N :_N  \langle\frak m \rangle \subseteq I^n N$, for all integer $n \geq 0$.
\end{cor}

\begin{proof}
As $I(I^n N :_N \langle \frak m \rangle) \subseteq I^{n+1} N:_N \langle\frak m\rangle$, for all integers $n \geq 0$, the claim follows from Lemma {\rm2.5}.
\end{proof}

\begin{defn}
Let $(R, \frak m)$ be a local (Noetherian) ring, $I$ an ideal of $R$ and $N$ an $R$-module. We define the $R$-module $\D(I,N)$ as the following:
\begin{center}
$\D(I,N):= \bigoplus _{n\geq 0} \D_{\frak m}(I^n N)$.
\end{center}
\end{defn}
As $\D_{\frak m}(.)$ is an $R$-linear and left exact functor, it follows that $\{ \D_{\frak m} (I^n N) \}_{n\geq 0}$ is a decreasing sequence and $I\D_{\frak m} (I^n N) \subseteq \D_{\frak m}(I^{n+1} N)$ for all integers $n \geq 0$. Hence $\D(I,N)$ is an $R[It]$-module, by Lemma {\rm2.5}.\\

\begin{lem}
Let $R$ be a Noetherian ring, $I$ an ideal of $R$ and $N$ a finitely generated $R$-module. Then the following conditions are equivalent:

$\rm(i)$ $D_I (N)$ is a finitely generated $R$-module.

$\rm(ii)$ For all $\frak p \in \Ass_R N$, the $R/ \frak p$-module $\D_{I(R/\frak p)} (R/\frak p)$ is finitely generated.
\end{lem}
\begin{proof}
See \cite[Lemma 3.3]{br2}.
\end{proof}

\begin{prop}
Let $(R, \frak m)$ be a local (Noetherian) ring, $I$ an ideal of $R$ and $N$ a finitely generated $R$-module. Then the following conditions are equivalent:

$\rm(i)$ $\D(I,N)$ is a finitely generated $R[It]$-module.

$\rm(ii)$ For all $\frak p \in \Ass_R N$, the module $\oplus_{n \geq 0} \D_{\frak m}(I^n+\frak p/ \frak p)$ is finitely generated over the Rees ring
 $R/\frak p[(I+\frak p / \frak p)t]$.
\end{prop}
\begin{proof}
In order to prove the implication $\rm(i)$ $\Longrightarrow$ $\rm(ii)$, suppose that $\frak p \in \Ass _R N$. Then in view of Proposition {\rm2.3}, there exists $\frak q \in \Ass_{R[It]} (\oplus_{n \geq 0} I^n N)$ such that $\frak q= \oplus_{n \geq 0} (I^n \cap \frak p)$. Since
$$\D(I,N)\cong \D_{\frak m} (\oplus_{n \geq 0} I^n N) \cong \D_{\frak m R[It]} (\oplus_{n \geq 0} I^n N)),$$ is a  finitely generated  $R[It]$-module, it follows from Lemma {\rm2.8} that the $R[It] / \frak q$-module $\D_{\frak m ({R[It]}/{\frak q})} (R[It] / \frak q)$ is finitely generated. Now, as
\begin{center}
$R[It] / \frak q \cong  R/\frak p[(I+\frak p /\frak p)t],$  and
$\D_{\frak m ({R[It]}/{\frak q})}({R[It]}/{•\frak q}) \cong \oplus_{n \geq 0} \D_{\frak m} ({I^n+ \frak p}/{\frak p•}),$
\end{center}
 we deduce that the $R/\frak p[({I+ \frak p}/{\frak p•})t]$-module
$\oplus_{n \geq 0} \D_{\frak m} ({I^n+ \frak p}/{\frak p•})$ is finitely generated.

Now, we show the conclusion $\rm(ii)$ $\Longrightarrow$ $\rm(i)$. To do this end, let $\frak q \in \Ass_{R[It]} N[It]$. Then, by virtue of Proposition {\rm2.3}, there exists $\frak p \in \Ass_R N$ such that $\frak q= \oplus_{n \geq 0} (I^n \cap \frak p)$. Since
\begin{center}
 ${R[It]}/{\frak q•} \cong R/\frak p[({I+ \frak p}/{\frak p•})t]$ and $\D_{\frak m ({R[It]}/{\frak q})} ({R[It]}/{\frak q•}) \cong \oplus_{n \geq 0} \D_{\frak m} ({I^n+ \frak p}/{\frak p•})$,
 \end{center}
  it follows from Lemma {\rm2.8} that the $R[It]$-module $\D_{\frak m R[It]}(N[It])$ is finitely generated, and so the $R[It]$-module $\oplus_{n \geq 0} \D_{\frak m} (I^n N)$ is finitely generated, as required.
\end{proof}

  The next proposition gives us a criterion for the finiteness of $R[It]$-module $D_{\frak m R[It]} (N[It])$, whenever $(R, \frak m)$ is a local ring and $N$ is a
  finitely generated module over $R$. To this end, let us, firstly, recall the important notion {\it analytic spread
of $I$ with respect to $N$}, over a local ring $(R, \frak{m}$),
introduced by Brodmann in \cite{B2}:
$$l(I,N):= \dim N[It]/ (\frak{m},u)N[It],$$ in
the case $N=R$, $l(I,N)$ is the classical analytic spread $l(I)$
of $I$, introduced by Northcott and Rees (see \cite{NR}).

\begin{prop}
Let $(R, \frak m)$ be a local (Noetherian) ring and $I$ an ideal of $R$. Let $N$ be a finitely generated $R$-module such that $l({I\widehat{R}+ \frak p}/{\frak p}) < \dim {\widehat{R}}/{\frak p•}$ for all $\frak p \in \Ass_{\widehat{R}} \widehat{N}$. Then the $R[It]$-module $D_{\frak m R[It]} (N[It])$ is finitely generated, and $\depth N > 0$.
\end{prop}
\begin{proof}
It is easy to see that
\begin{center}
$D_{\frak m R[It]} (N[It]) \otimes _{R[It]} \widehat{R}[(I\widehat{R})t] \cong D_{\frak m \widehat{R}[(I\widehat{R})t]} (\oplus_{n \geq 0} I^n \widehat{N})$,
\end{center}
and so by faithfully flatness of $\widehat{R}[(I\widehat{R})t]$ over $R[It]$, it is enough for us to show that the $\widehat{R}[(I\widehat{R})t]$-module $\oplus _{n \geq 0} \D_{\frak m\widehat{R}[(I\widehat{R})t]} (I^n \widehat{N})$ is finitely generated. In order to do this, in view of Proposition {\rm2.9}, it is enough to show that $\oplus_{n\geq 0} \D_{\frak m} (I^n \widehat{R} + \frak p/\frak p)$ is finitely generated over $\widehat{R}/\frak p[(I\widehat{R}+ \frak p/\frak p)t]$ for all $\frak p \in \Ass_{\widehat{R}} \widehat{N}$. But this follows easily from \cite[Proposition]{sc1} and the assumption $l(I\widehat{R}+ \frak p/\frak p) < \dim \widehat{R}/\frak p$.
\end{proof}

\begin{rem}
Before bringing the next result we fix a notation, which is employed by  P. Schenzel in \cite{sc2} in the case $N=R$. Let $S$ be a multiplicatively closed subset of a Noetherian ring $R$. For a submodule $M$ of a finitely generated $R$-module $N$, we use $S(M)$ to denote the submodule $\bigcup_{s \in S} (M:_{N} s)$. Note that the primary decomposition of $S(M)$ consists of the intersection of all primary components of $M$ whose associated prime ideals do not meet $S$. In other words
\begin{center}
$\Ass _{R} {N}/{S(M)}= \{ \frak p \in \Ass _{R} N/M: \frak p \cap S = \emptyset \}$.
\end{center}

In particular, if $S=R \setminus \bigcup \{ \frak p \in \mAss_R {N}/{IN} \}$, then for any $n \in \mathbb{N}$, $S(I^n N)$ is denoted by $(IN)^{(n)}$, where $I$ is an ideal of $R$.
\end{rem}

The following lemma is needed in the proof of Theorem 2.13.

\begin{lem}
Let $R$ be a Noetherian ring and $N$ an $R$-module. Let $M$ and $L$ be two submodules of $R$ such that $M_{\frak p} \subseteq L_{\frak p}$ for all $\frak p \in \Ass _R {N}/{L•}$. Then $M \subseteq L$.
\end{lem}

\begin{proof}
The assertion follows from the fact that $\Ass _R ({M+L}/{L•}) \subseteq \Ass_R {N}/{L•}$.
\end{proof}

Following, we investigate a fundamental characterization for  linearly equivalence between the $I$-adic and symbolic topologies on a
 finitely generated $R$-module $N$,  for certain  ideal $I$ of $R$.
 This result plays a key role in the proof of the main theorem.

To this end, recall that,  in \cite{ra2}, L.J. Ratliff, Jr., (resp. in \cite{br1} Brodmann)  introduced the interesting set of
associated primes $\bar{A^*}(I):= \Ass_R R/(I^n)_a $  (resp.  $A^* (I, N):= \Ass_R {N}/{I^n N}$),  for large $n$.
Here $I_a$ denotes the integral closure of $I$ in $R$, i.e., $I_a$ is the ideal of $R$ consisting of all
elements $x\in R$ which satisfy  an equation $x^n+ r_1x^{n-1}+\cdots + r_n= 0$, where $r_i\in I^i, i=1, \ldots, n$.

Moreover, recall that  a local ring $(R, \frak{m})$   is said to be a
{\it quasi-unmixed ring} if for every $\frak{p}\in \mAss \widehat{R}$, the condition $\dim \widehat{R}/\frak{p}= \dim R$ is satisfied.

\begin{thm}
Let $R$ be a Noetherian ring, $I$ an ideal of $R$ and let $N$ be a finitely generated  $R$-module such that $E(I,N)=\mAss _R {N}/{IN•}$. Then, the $I$-adic topology, $\{ I^n N \}_{n \geq 0}$ and the topology defined by the filtration $\{( I N)^{(n)} \}_{n \geq 0}$ are linearly equivalent.

\end{thm}

\begin{proof}
Let $\frak q \in A^* (I,N) \setminus \mAss_R N/IN$ and let $z \in \Ass_{\widehat{R_{\frak q}}} \widehat{N_{\frak q}}$. Then, by assumption, $\frak q \notin E(I,N)$. Hence, in view of \cite[Lemma 3.2]{Ah}, $\frak q R_{\frak q} \notin E(IR_{\frak q}, N_{\frak q})$, and so it follows from \cite[Proposition 3.6]{Ah} that $\frak q \widehat{R_{\frak q}}/z \notin E(I\widehat{R_{\frak q}}+z / z)$. Thus by virtue of \cite[Lemma 2.1]{mc2}, $\frak q \widehat{R_{\frak q}}/z \notin \bar{A^*}(I\widehat{R_{\frak q}}+z/z)$. As $\widehat{R_{\frak q}}/z$ is quasi-unmixed, it follows from McAdam's result \cite[Proposition 4.1]{mc1} that
$$l(I\widehat{R_{\frak q}}+z/{z•}) < \dim \widehat{R_{\frak q}}/{z}.  \quad \quad \quad  \quad  \quad  \quad  (\dag)$$

Now, we show that there exists a non-negative integer $k$ such that $(IN)^{(n+k)} \subseteq I^n N$ for all integers $n \geq 0$. To do this, it is easy to see that, $(IN)^{(s)} _{\frak p} \subseteq (I^s N)_{\frak p}$ for all $\frak p \in \mAss_R {N}/{IN}$ and for all integers $s\geq 0$. Moreover, if for every $\frak q \in A^* (I,N) \setminus \mAss_R {N}/{IN}$ there exists an integer $k_{\frak q}$ such that
\begin{center}
$(IN)^{(n +k_\frak q)} _{\frak q} \subseteq (I^n N)_{\frak q}$,
\end{center}
then by considering $$k := \max \{ k_{\frak q} : \frak q \in A^* (I,N) \setminus \mAss_R {N}/{IN} \},$$ one easily sees that $(IN)^{(n+k)} \subseteq I^n N$. Since both $A^* (I,N)$ and $\mAss_R {N}/{IN}$ behave well under localization, we may assume by localizing at $\frak q$ that $(R, \frak m)$ is a local ring.

 Now, we use induction on $\dim {N}/{IN•}:=d$. It is clear that $d \geq 1$. Now, if $d=1$, then, as $\Ass_R {N}/{IN} \subseteq \Supp {N}/{IN•}$ and $\frak m \in \Supp {N}/{IN•}$ it follows that the only possible embedded prime of $\Ass_R {N}/{IN}$ is $\frak m$, and so in view of Remark {\rm2.1} we have $$I^s N:_N \langle \frak m \rangle = (IN)^{(s)}$$  for all integers $s \geq 0$. Next, it follows from $(\dag)$ and Proposition {\rm2.10} that the $R[It]$-module $\oplus_{ n\geq 0} \D_{\frak m}(I^n N)$ is finitely generated  and $\depth N> 0$. Hence in view of Lemma {\rm2.2}, the module $\oplus_{ n\geq 0} (I^n N :_N \langle \frak m \rangle)$ is finitely generated over the Rees ring $R[It]$, and so by virtue of Corollary {\rm2.6}, there exists an integer $t$ such that $I^{n+t} N :_N \langle \frak m \rangle \subseteq I^n N$ for all integers $n \geq 0$. Therefore $(IN)^{(n+k)} \subseteq I^n N$, and so the result holds for $d=1$.

We therefore assume, inductively, that $d> 1$ and the result has been proved for smaller values of $d$. If $\frak q \neq \frak m$ and $\frak q \in A^* (I,N)$, then
\begin{center}
$\dim {N_{\frak q}}/{I N_{\frak q}} = \Ht _{{N}/{IN•}} \frak q < \Ht_{{N}/{IN•}} \frak m =\dim {N}/{IN•} = d$.
\end{center}
Hence by induction hypothesis, there exists a non-negative integer $k_{\frak q}$ such that $$(IN)^{(n +k_\frak q)} _{\frak q} \subseteq (I^n N)_{\frak q},$$ for all integers $n \geq 0$. Now, in view of Remark {\rm2.1},
\begin{center}
$\Ass _R {N}/{(I^n N:_N \langle \frak m \rangle)} = \Ass_R {N}/{I^n N} \setminus V(\frak m)$,
\end{center}
it follows that for all $\frak q \in \Ass _R {N}/{(I^n N:_N \langle \frak m \rangle)}$, there exists a non-negative integer $k_{\frak q}$ such that
\begin{center}
$(IN)^{(n +k_\frak q)} _{\frak q} \subseteq (I^n N)_{\frak q} \subseteq ((I^n N)_{\frak q}:_{N_{\frak q}} \langle \frak m R_{\frak q} \rangle)$,
\end{center}
for all integers $n \geq 0$. Hence by considering
\begin{center}
$k:= \max \{ k_{\frak q} : \frak q \in \Ass _R {N}/({I^n N:_N \langle \frak m \rangle}) \}$,
\end{center}
we get $$(IN)^{(n +k)} _{\frak q} \subseteq (I^n N)_{\frak q}:_{N_{\frak q}} \langle\frak m R_{\frak q} \rangle,$$ for all $\frak q \in \Ass _R {N}/({I^n N:_N \langle \frak m \rangle})$ and all integers $n \geq 0$. Therefore, by virtue of the Lemma {\rm2.12}, we have $$(IN)^{(n +k)} \subseteq (I^n N):_{N} \langle \frak m \rangle.$$

 On the other hand, in view of Corollary {\rm2.6}, there exists an integer $s \geq 0$ such that
  $${I^{n+s} N:_N \langle \frak m \rangle} \subseteq I^n N$$ for all integers $n\geq0$. Consequently
  $$(IN)^{(n+k+s)} \subseteq I^{n+s}N:_{N} \langle \frak m \rangle \subseteq I^n N,$$ for all integers $n \geq 0$, and thus the topologies defined by the filtrations $\{ I^n N \}_{n \geq 0}$ and $\{ (IN)^{(n)} \}_{n\geq0}$ are linearly equivalent.
\end{proof}

We are now ready to state and prove the main theorem of this paper, which is a new characterization of locally unmixed modules in terms of comparison of the topologies defined be certain decreasing families of a submodules of finitely generated modules over a commutative Noetherian ring.  One of the implications in the proof of this theorem follows from \cite[Theorem 3.2]{na}.

\begin{thm}
Let $R$ be a Noetherian ring and $N$ a non-zero finitely generated  $R$-module. Then the following conditions are equivalent:

$\rm(i)$ $N$ is locally unmixed.

$\rm(ii)$ For any $N$-proper ideal $I$ of $R$ generated by $\Ht_N I$ elements, the $I$-adic topology is linearly equivalent to the symbolic topology.
\end{thm}

\begin{proof}
The implication $\rm(ii)$ $\Longrightarrow$ $\rm(i)$  follows easily from \cite[Theorem 3.2]{na}. In order to prove the conclusion $\rm(i)$ $\Longrightarrow$ $\rm(ii)$, let $I$ be an $N$-proper ideal of $R$ which is generated by $\Ht_{N} I$ elements. Then, in view of Theorem {\rm2.13} it is enough for us to show that $E(I,N)=\mAss_R {N}/{IN}$. Suppose that $\frak p \in E(I,N)$, and we show that $\frak p \in \mAss _R {N}/{IN}$. Let $\Ht_N I:=n$. Then by \cite[Theorem 2.1]{na}, there exist the elements $x_1, \dots, x_n$ in $I$ such that $\Ht_N (x_1,\dots, x_i)=i$ for all $1 \leq i \leq n$. As, in view of \cite[Corollary 3.11]{na}, $x_1,\dots, x_n$ is an essential sequence on $N$, and the fact that ${\rm egrade}(I,N) \leq \Ht_N I$, it follows that ${\rm egrade}(I,N)=n$. Now, analogous to the proof of \cite[Theorem 125]{kap}, it is easy to see that $I$ can be generated by an essential sequence of length n. Therefore by \cite[Lemma 3.8]{na}, we have $\frak p \in \mAss_R {N}/{IN}$,  and so $E(I,N)\subseteq\mAss_R {N}/{IN}$. As the opposite inclusion is obvious, the result follows.
\end{proof}

\begin{center}
{\bf Acknowledgments}
\end{center}
The authors are deeply grateful to the referee for his/her careful
reading of the paper and valuable suggestions. Also, we  would like to thank Professors M.P. Brodmann and S. Goto for their useful comments on Theorem 2.13.

\end{document}